\documentclass{article}
\usepackage{amssymb}
\usepackage{amsmath}
\usepackage{amsthm}

\newtheorem{theorem}{Theorem}[section]
\newtheorem{lemma}[theorem]{Lemma}
\newtheorem{proposition}[theorem]{Proposition}
\newtheorem{corollary}[theorem]{Corollary}

\theoremstyle{definition}
\newtheorem{definition}[theorem]{Definition}
\newtheorem{example}[theorem]{Example}
\newtheorem{algorithm}[theorem]{Algorithm}

\theoremstyle{remark}
\newtheorem{remark}[theorem]{Remark}

\numberwithin{equation}{section}

\begin{document}
	
%\title[Numerical semigroups and economic incentives]{Numerical semigroups in a problem about economic incentives to consumers}
\title{Numerical semigroups in a problem about economic incentives for consumers}

\author{Aureliano M. Robles-P\'erez\thanks{Both of the authors are supported by FQM-343 (Junta de Andaluc\'{\i}a), MTM2014-55367-P (MINECO, Spain), and FEDER funds. The second author is also partially supported by Junta de Andaluc\'{\i}a/FEDER Grant Number FQM-5849.} \thanks{Departamento de Matem\'atica Aplicada, Universidad de Granada, 18071-Granada, Spain. \newline E-mail: {\bf arobles@ugr.es}}
	\mbox{ and} Jos\'e Carlos Rosales$^*$\thanks{Departamento de \'Algebra, Universidad de Granada, 18071-Granada, Spain. \newline E-mail: {\bf jrosales@ugr.es}} }

\date{ }

\maketitle

\begin{abstract}
Motivated by a promotion to increase the number of musical downloads, we introduce the concept of $C$-incentive and show an algorithm that compute the smallest $C$-incentive containing a subset $X \subseteq {\mathbb N}$. On the other hand, in order to study $C$-incentives, we see that we can focus on numerical $C$-incentives. Then, we establish that the set formed by all numerical $C$-incentives is a Frobenius pseudo-variety and we show an algorithmic process to recurrently build such a pseudo-variety.
\end{abstract}

\medskip


\section{Introduction}

A certain commercial music streaming service designs a new promotion for one month. Namely, depending on the demand of a song, the cost of the download is 5, 7, 9 or 11 cents. In addition, if the customer waits
\begin{itemize}
\item less than one hour between two downloads, then there is a discount of 3 cents in the second one;
\item more than two hours between two downloads, then there is an additional charge of 2 cents in the second one.
\end{itemize}
For instance, suppose a customer buys a song by 7 cents, then thirty minutes later gets a discount of 3 cents when buying a 9 cent song. Moreover, four hours later, he has an additional charge of 2 cents when purchasing a 5 cent song; and so forth. Our purpose is to study the set ${\mathcal F}$ formed by the amounts that can appear in the customers' invoices at the end of the promotion.

It is clear that we can associate each customer with an odd finite length list $(x_1,x_2,\ldots,x_n)$ such that $x_1,x_3,\ldots,x_n \in \{5,7,9,11\}$, $x_2,x_4,\ldots,x_{n-1} \in \{-3,0,2\}$, and the invoice is $x_1+x_2+\cdots+x_n$. Therefore,
	$${\mathcal F} = \left\{x_1+\cdots+x_n \mid n \mbox{ is an odd positive integer}, x_1,x_3,\ldots,x_n \in \{5,7,9,11\}, \right.$$
	$$\left.  \mbox{ and } x_2,x_4,\ldots,x_{n-1} \in \{-3,0,2\} \right\} \cup \{0\}.$$

The abstraction of the previous example leads us to the following definition: let $A,B$ be two non-empty sets of integers. An \emph{$(A,B)$-sequence} is an odd finite length list $(x_1,x_2,\ldots,x_n)$ such that $x_1,x_3,\ldots,x_n \in A$ and $x_2,x_4,\ldots,x_{n-1} \in B$. Observe that, if we denote by $|(x_1,\ldots,x_n)|=x_1+\cdots+x_n$ and by ${\rm M}(A,B)=\{|x| \mid x \mbox{ is a $(A,B)$-sequence} \} \cup \{0\}$, then ${\mathcal F}={\rm M}(\{5,7,9,11\},\{-3,0,2\})$.

In the remainder of the introduction, we are going to suppose that $A$ is a non-empty finite set of positive integers, that $B$ is a finite set of integers which contain the zero element, and that
$\min(A)+\min(B)\geq 0$. Moreover, as usual, by ${\mathbb Z}$ and ${\mathbb N}$ we denote the set of integers and the set of non-negative integers, respectively.

Firstly, we will show in Section~\ref{first-results} that ${\rm M}(A,B)$ is a submonoid of $({\mathbb N},+)$. We will observe that $A \subseteq {\rm M}(A,B)$ and that, if $b\in B$, then could be that $b\not\in {\rm M}(A,B)$; however, if $x,y \in {\rm M}(A,B) \setminus \{0\}$, then $x+y+b \in {\rm M}(A,B)$. This fact will allow us to give the concept of \emph{$C$-incentive}: if $C$ be a subset of ${\mathbb Z}$, then a $C$-incentive is a submonoid $M$ of $({\mathbb N},+)$ such that $\{x+y\}+C \subseteq M$ for all $x,y \in M\setminus \{0\}$. In Section~\ref{first-results} we will also see that ${\rm M}(A,B)$ is the smallest (with respect to inclusion) $(B \setminus \{0\})$-incentive containing $A$. In this way, following with our example, we have that ${\mathcal F}$ is the smallest $\{-3,2\}$-incentive containing the set $\{5,7,9,11\}$.

From this point to the end of the introduction, we will suppose that $C$ is a non-empty finite set of ${\mathbb Z}$. In Section~\ref{smallest-incentive} we will approach the problem of computing the smallest $C$-incentive which contains a given set $X$ of non-negative integers. In order to see that the above mentioned set exists, we will show the conditions that $X$ has to satisfy with respect to $C$. Once that is done, we will show an algorithm to compute the smallest $C$-incentive in the case that it exists.

A \emph{numerical semigroup} is a submonoid of $({\mathbb N},+)$, $S$, such that $\gcd(S)=1$. From this concept we have the following one: if $M$ is a $C$-incentive, then $M$ is \emph{numerical} in case that $\gcd(M)=1$. We will denote by ${\rm I}(C)=\{M \mid M \mbox{ is a $C$-incentive}\}$ and by ${\rm NI}(C)=\{M \mid M \mbox{ is a numerical $C$-incentive}\}$. In Section~\ref{numerical-incentives} we will show that ${\rm I}(C) \setminus \{0\} = \bigcup_{d \in D} \left\{ dS \mid S \in {\rm NI}\left(\frac{C}{d}\right) \right\}$, where $D$ is the set of all positive divisors of $\gcd(C)$. Observe that this result points out that for studying $C$-incentives we can focus on numerical $C$-incentives.

In \cite{variedades} it was introduced the concept of Frobenius varieties, in order to unify several results which appeared in \cite{patterns, systems, arf, saturated}. Nevertheless, there exist families of numerical semigroups that are not Frobenius varieties. For instance, the family of numerical semigroups with maximal embedding dimension and fixed multiplicity (see \cite{med}). The study, in \cite{bagsvo}, of this class of numerical semigroups led to the concept of $m$-variety. In turn, in \cite{pseudovar} was introduced the concept of Frobenius pseudo-variety which generalizes the concepts of Frobenius variety and $m$-variety.

In Section~\ref{pseudo-variety} we will show that ${\rm NI}(C)$ is a Frobenius pseudo-variety. This fact, together with several results of \cite{pseudovar}, allows us to arrange the elements of ${\rm NI}(C)$ in a tree with root. Then, in Section~\ref{tree} we will show a procedure to recursively build ${\rm NI}(C)$. In order to give it, we will need to describe how are the children of a vertex in the tree associated to ${\rm NI}(C)$.

In the end, in Section~\ref{suggestion} we will study the tree of numerical $C$-incentives containing a given set $X$. In particular, we will determine when that tree is finite and, consequently, we can draw a whole tree.

To finish this introduction, we review some works which have driven us to the study of the $C$-incentives.

A $(v,b,r,k)$-configuration (see \cite{configuration}) is a connected bipartite graph with $v$ vertices on one side, each of them of degree $r$, and $b$ vertices on the other side, each of them of degree $k$, and with no cycle of length 4. A $(v,b,r,k)$-configuration can also be seen as a combinatorial configuration (see \cite{stokes-bras}) with $v$ points, $b$ lines, $r$ lines through every point and $k$ points on every line. It is said that the tuple $(v,b,r,k)$ is configurable if a $(v,b,r,k)$-configuration exists. In \cite{configuration} was shown that, if $(v,b,r,k)$ is configurable, then $vr=bk$ and, consequently, there exists $d$ such that $v=d\frac{k}{\gcd\{r,k\}}$ and $b=d\frac{r}{\gcd\{r,k\}}$. The main result of \cite{configuration} states that, if $k,r$ are integers greater than or equal to 2, then
	$$S_{(r,k)} = \left\{ d \in {\mathbb N} \;\Big\vert \left(d\frac{k}{\gcd\{r,k\}}, d\frac{r}{\gcd\{r,k\}}, r,k\right) \mbox{ is configurable} \right\}$$
is a numerical semigroup. Moreover, in \cite{stokes-bras} was proved that, if a configuration is balanced (that is, $r=k$), then $\left\{x+y-1,x+y+1\right\} \subseteq S_{(r,r)}$, for all $x,y \in S_{(r,r)} \setminus \{0\}$. Therefore, $S_{(r,r)}$ is a numerical $\{-1,1\}$-incentive.

On the other hand, $\{1\}$-incentives, $\{-1\}$-incentives, $\{-1,1\}$-incentives, and $C$-incentives with $C \subseteq {\mathbb N}$ are studied in \cite{frases}, \cite{digitales}, \cite{acotados, benefits}, and \cite{brazaletes}, respectively.

\section{First results}\label{first-results}

In this section, $A$ will denote a non-empty finite set of positive integers, and $B$ will be a finite subset of ${\mathbb Z}$ such that $0 \in B$. Moreover, we will suppose that $\min(A)+\min(B) \geq 0$.

Along this work, an \emph{$(A,B)$-sequence} is an odd finite length list $(x_1,x_2,\ldots,x_n)$ such that $x_1,x_3,\ldots,x_n \in A$ and $x_2,x_4,\ldots,x_{n-1} \in B$. In addition, we denote by $|(x_1,\ldots,x_n)|=x_1+\cdots+x_n$ and by ${\rm M}(A,B)=\{|x| \mid x \mbox{ is a $(A,B)$-sequence} \} \cup \{0\}$. Our first purpose will be to show that ${\rm M}(A,B)$ is a submonoid of $\left({\mathbb N},+\right)$. The following result has an immediate proof.

\begin{lemma}\label{lem1}
If $(x_1,\ldots,x_n)$ is an $(A,B)$-sequence and $n\geq 3$, then $(x_3,\ldots,x_n)$ is an $(A,B)$-sequence.
\end{lemma}

\begin{lemma}\label{lem2}
If $(x_1,\ldots,x_n)$ is an $(A,B)$-sequence, then $|(x_1,\ldots,x_n)| \in {\mathbb N} \setminus \{0\}$.
\end{lemma}

\begin{proof}
Let $m=|(x_1,\ldots,x_n)|$. By induction over $n$, we will see that $m \in {\mathbb N} \setminus \{0\}$. First, if $n=1$, then $m=x_1 \in A \subseteq {\mathbb N} \setminus \{0\}$. Now, let us suppose that $n\geq 3$ (remember that $n$ is odd). Since $x_1\in A$, $x_2 \in B$, and $\min(A)+\min(B)\geq0$, then $x_1+x_2 \in {\mathbb N}$. By hypothesis of induction and Lemma~\ref{lem1}, we have that $|(x_3,\ldots,x_n)| \in {\mathbb N} \setminus \{0\}$. Therefore, $m=|(x_1,\ldots,x_n)|=x_1+x_2+|(x_3,\ldots,x_n)| \in {\mathbb N} \setminus \{0\}$.
\end{proof}

\begin{proposition}\label{prop3}
${\rm M}(A,B)$ is a submonoid of $({\mathbb N},+)$.
\end{proposition}

\begin{proof}
By Lemma~\ref{lem2}, we know that ${\rm M}(A,B) \subseteq {\mathbb N}$. Since $0 \in {\rm M}(A,B)$, in order to finish the proof, it is enough to see that ${\rm M}(A,B)$ is closed for the addition. So, let $s,t \in {\rm M}(A,B) \setminus \{0\} $. Then there exist two $(A,B)$-sequences, $(x_1,\ldots,x_n)$ and $(y_1,\ldots,y_n)$, such that $|(x_1,\ldots,x_n)|=s$ and $|(y_1,\ldots,y_n)|=t$. It is obvious that $(x_1,\ldots,x_n,0,y_1,\ldots,y_n)$ is an $(A,B)$-sequence with $|(x_1,\ldots,x_n,0,y_1,\ldots,y_n)| = s+t$. Thus, $s+t \in {\rm M}(A,B)$.  
\end{proof}

Let us observe that $A \subseteq {\rm M}(A,B)$. However, if $b\in B$, then it is possible that $b\notin {\rm M}(A,B)$. Despite this situation, we have the next result.

\begin{lemma}\label{lem4}
If $s,t \in {\rm M}(A,B)\setminus \{0\}$ and $b\in B$, then $s+t+b \in {\rm M}(A,B)$.
\end{lemma}

\begin{proof}
Following the proof of Proposition~\ref{prop3}, let $(x_1,\ldots,x_n)$, $(y_1,\ldots,y_n)$ be two $(A,B)$-sequences such that $|(x_1,\ldots,x_n)|=s$, $|(y_1,\ldots,y_n)|=t$. Obviously, $(x_1,\ldots,x_n,b,y_1,\ldots,y_n)$ is an $(A,B)$-sequence with $|(x_1,\ldots,x_n,b,y_1,\ldots,y_n)|$ $=s+t+b$.
\end{proof}

The previous results lead us to give the following definition.

\begin{definition}\label{def1}
Let $C$ be a subset of ${\mathbb Z}$. We will say that a submonoid $M$ of $({\mathbb N},+)$ is a \emph{$C$-incentive} if it verifies that $\{s+t\} + C \subset M$ for all $s,t \in M \setminus \{0\}$.
\end{definition}

\begin{theorem}\label{thm5}
${\rm M}(A,B)$ is the smallest (with respect to inclusion) $B$-incentive containing $A$.
\end{theorem}

\begin{proof}
By Proposition~\ref{prop3} and Lemma~\ref{lem4}, we have that ${\rm M}(A,B)$ is a $B$-incentive containing $A$. Let $T$ be a $B$-incentive containing $A$. Let us see that ${\rm M}(A,B) \subseteq T$. If $m \in {\rm M}(A,B) \setminus \{0\}$, then there exists an $(A,B)$-sequence $(x_1,\ldots,x_n)$ such that $|(x_1,\ldots,x_k)|=m$. By induction over $k$, we will show that $m \in T$. If $k=1$, then $m=x_1\in A \subset T$. Now, we can suppose that $k \geq 3$. By hypothesis of induction and Lemma~\ref{lem1}, we know that $x_1,|(x_3,\ldots,x_k)| \in T$. Since $T$ is a $B$-incentive and $x_2\in B$, we have that $x_1+|(x_3,\ldots,x_k)|+x_2 \in T$. Therefore, $m=|(x_1,\ldots,x_k)|\in T$.
\end{proof}

The following result is easy to prove.

\begin{proposition}\label{prop6}
Let $C$ be a subset of ${\mathbb Z}$ and let $M$ be a submonoid of $({\mathbb N},+)$. Then $M$ is a $C$-incentive if and only if $M$ is a $(C\setminus \{0\})$-incentive.
\end{proposition}

As a consequence of the above results, and considering the example of the introduction, we can assert that ${\mathcal F}$ is the smallest $\{-3,2\}$-incentive containing the set $\{5,7,9,11\}$.

In the next section we will study how to compute the smallest $C$-incentive that contains a given set of positive integers. Now, to finish this section, we will give a result that allows us to decide whether or not a submonoid of $({\mathbb N},+)$ is a $C$-incentive.

Let $X$ be a non-empty subset of ${\mathbb N}$. We will denote by $\langle X \rangle$ the submonoid of $({\mathbb N},+)$ generated by $X$, that is,
	$$\langle X \rangle = \left\{ \lambda_1x_1+\cdots+\lambda_nx_n \mid n \in {\mathbb N} \setminus \{0\}, \; x_1,\ldots,x_n \in X, \; \lambda_1,\ldots,\lambda_n \in {\mathbb N} \right\}.$$
If $M=\langle X \rangle$, then we will say that $M$ is generated by $X$ or, equivalently, that $X$ is a \emph{system of generators of} $M$. Moreover, if $M \not= \langle Y \rangle$ for all $Y \subsetneqq X$, then we will say that $X$ is a \emph{minimal system of generators of} $M$. The following result is \cite[Corollary 2.8]{springer}.
	
\begin{lemma}\label{lem7}
Let $M$ be a submonoid of $({\mathbb N},+)$. Then $M$ has a unique minimal system of generators. In addition, such a system is finite.
\end{lemma}

If $M$ is a submonoid of $({\mathbb N},+)$, then we will denote by ${\rm msg}(M)$ the minimal system of generators of $M$. It is easy to show (see \cite[Lemma 2.3]{springer}) that ${\rm msg}(M) = (M \setminus \{0\}) \setminus ( (M \setminus \{0\}) + (M \setminus \{0\}) )$.

\begin{proposition}\label{prop8}
Let $C$ be a non-empty subset of ${\mathbb Z}$ and let $M$ be a submonoid of $({\mathbb N},+)$ generated by the set of positive integers $\{n_1,\ldots,n_p\}$. Then $M$ is a $C$-incentive if and only if $\{n_i+n_j\}+C \subset M$ for all $i,j \in \{1,\ldots,p\}$.
\end{proposition}

\begin{proof}
The necessary condition is trivial. Let us see the sufficient condition. Let $x,y \in M \setminus \{0\}$ and let $c \in C$. By the comment above this proposition, we know that there exist $i,j \in \{1,\ldots,p\}$ and $s,t \in M$ such that $x=n_i+s$ and $y=n_j+t$. Thereby, $x+y+c = (n_i+n_j+c)+s+t \in M$. Therefore, $M$ is a $C$-incentive.
\end{proof}

Let us see an example to illustrate the previous proposition.

\begin{example}\label{exmp9}
We have that $\{3,7,8\} + \{3,7,8\} + \{-3,2\} \subseteq \langle 3,7,8 \rangle$. Consequently, by Proposition~\ref{prop8}, we can assert that $\langle 3,7,8 \rangle$ is a $\{-3,2\}$-incentive.
\end{example}

\section{An algorithm for finding the smallest $C$-incen\-tive containing a given set of positive integers}\label{smallest-incentive}

Let $C$ be a subset of ${\mathbb Z}$. We will say that $X \subseteq {\mathbb N}$ is a \emph{$C$-admissible set} if there exists at least a $C$-incentive containing it. We begin this section by characterizing the $C$-admissible sets. Then, if $X$ is a $C$-admissible set, we will show that there exists the smallest $C$-incentive containing it. Finally, we will give an algorithm to computing it.

First of all, let us observe that sometimes there is not any $C$-incentive containing $X$, such as it is shown in the following example.

\begin{example}\label{exmp10}
Let us see that there does not exist any $\{-4\}$-incentive containing the set $\{3\}$. In fact, by contradiction, let us suppose that $M$ is a $\{-4\}$-incentive containing $\{3\}$. Then $2=3+3-4 \in M$. Therefore, $1=2+3-4 \in M$. Consequently, $-2=1+1+-4 \in M$.
\end{example}

In order to characterize the $C$-admissible sets, we need three lemmas. From now on, we are going to suppose that $C$ is a non-empty finite set of ${\mathbb Z}$ and we denote by $\theta(C)=-\min(C \cup \{0\})$.

\begin{lemma}\label{lem11}
The set $S=\{0,\theta(C),\to\} = \{0\} \cup \{n \in {\mathbb N} \mid n\geq \theta(C)\}$ is a $C$-incentive.
\end{lemma}

\begin{proof}
It is clear that $S$ is a submonoid of $({\mathbb N},+)$. Let $a,b \in S\setminus \{0\}$ and $c\in C$. Then $a+b+c \geq \theta(C)$ and, therefore, $a+b+c \in S$. 
\end{proof}

\begin{lemma}\label{lem12}
If $M$ is a $C$-incentive, then $M \subseteq \{0,\theta(C),\to\}$ or $M=\left\langle \frac{\theta(C)}{2}\right\rangle$.
\end{lemma}

\begin{proof}
If $\theta(C)=0$, then it is obvious that $M\subseteq \{0,\to\}$. Thereby, we can suppose that $\theta(C)=-c>0$ for some $c \in C$. Let $m$ be the least positive integer belonging to $M$. If $M \not\subseteq \{0,\theta(C),\to\}$, then $m < \theta(C)$. Since $m+m-\theta(C)<m$ and $m+m-\theta(C)=m+m+c \in M$, then $m+m-\theta(C)=0$ and, therefore, $m=\frac{\theta(C)}{2}$. Consequently, $\left\langle \frac{\theta(C)}{2}\right\rangle \subseteq M$. Now, let us see that $M \subseteq \left\langle \frac{\theta(C)}{2}\right\rangle$. If it is not the case, then there exists $w=\min\left\{ x\in M \mid x \not\equiv 0 \pmod{\frac{\theta(C)}{2}} \right\}$. Then, $w-\frac{\theta(C)}{2}=w+\frac{\theta(C)}{2}-\theta(C) = w+m+c \in M$, in contradiction with the definition of $w$. 
\end{proof}

\begin{lemma}\label{lem13}
The monoid $T=\left\langle \frac{\theta(C)}{2}\right\rangle$ is a $C$-incentive if and only if $C \subseteq \left\{k\frac{\theta(C)}{2} \mid  k \in \{-2,-1\} \cup {\mathbb N} \right\}$.
\end{lemma}

\begin{proof}
By Proposition~\ref{prop8}, $T$ is a $C$-incentive if and only if $\left\{\frac{\theta(C)}{2}+\frac{\theta(C)}{2}\right\}+C \subseteq T$. That is, $T$ is a $C$-incentive if and only if, for every $c \in C$, there exists $k_c \in {\mathbb N}$ such that $\theta(C)+c=k_c\,\frac{\theta(C)}{2}$. From this equality, the conclusion is clear.
\end{proof}

\begin{proposition}\label{prop14}
Let $X$ be a subset of ${\mathbb N}$. Then $X$ is $C$-admissible if and only if $X \subseteq \{0,\theta(C),\to \}$ or $X \subseteq \left\langle \frac{\theta(C)}{2}\right\rangle$ and $C \subseteq  \left\{k\frac{\theta(C)}{2} \mid  k \in \{-2,-1\} \cup {\mathbb N} \right\}$.
\end{proposition}

\begin{proof}
From Lemmas~\ref{lem12} and \ref{lem13}, we have the necessary condition. For the sufficient condition, we apply Lemmas~\ref{lem11} and \ref{lem13}.
\end{proof}

The next result has an immediate proof and, therefore, we omit it.

\begin{lemma}\label{lem15}
The intersection of $C$-incentives is a $C$-incentive.
\end{lemma}

This lemma leads us to the following definition.

\begin{definition}
Let $X$ be a $C$-admissible set. Let ${\rm L}_C(X)$ be the intersection of all $C$-incentives containing $X$. We say that ${\rm L}_C(X)$ is the \emph{$C$-incentive generated by} $X$.
\end{definition}

As a consequence of Lemma~\ref{lem15}, we have that ${\rm L}_C(X)$ is the smallest (with respect the inclusion) $C$-incentive containing $X$.

Let us denote by ${\rm I}(C) = \{ M \mid M \mbox{ is a $C$-incentive} \}$. 

\begin{theorem}\label{thm16}
With the above notation we have the following.
\begin{enumerate}
\item If $\,C \not\subseteq \left\{k\frac{\theta(C)}{2} \mid  k \in \{-2,-1\} \cup {\mathbb N} \right\}$, then $${\rm I}(C) = \left\{ {\rm L}_C(X) \mid X \mbox{ is a finite subset of }\, \{0,\theta(C),\to \} \right\}.$$
\item If $\,C \subseteq \left\{k\frac{\theta(C)}{2} \mid  k \in \{-2,-1\} \cup {\mathbb N} \right\}$, then $${\rm I}(C) = \left\{ {\rm L}_C(X) \mid X \mbox{ is a finite subset of }\, \{0,\theta(C),\to \} \right\} \cup  \mbox{$\left\langle \frac{\theta(C)}{2} \right\rangle$}.$$
\end{enumerate}
\end{theorem}

\begin{proof}
Let us observe that, if $M \in {\rm I}(C)$, then $M$ is a submonoid of $({\mathbb N},+)$ and, by Lemma~\ref{lem7}, there exists a finite subset $X$ of ${\mathbb N}$ such that $M=\langle X \rangle$ and, moreover, $M={\rm L}_C(X)$. Now, by Lemmas~\ref{lem12} and \ref{lem13}, we have that, if $X \not\subseteq \{0,\theta(C),\to\}$, then $X \subseteq \left\langle \frac{\theta(C)}{2} \right\rangle$, $C \subseteq \left\{k\frac{\theta(C)}{2} \mid  k \in \{-2,-1\} \cup {\mathbb N} \right\}$, and, consequently, ${\rm L}_C(X) = \left\langle \frac{\theta(C)}{2} \right\rangle$.  
\end{proof}

Let us observe that $\emptyset$ is a $C$-admissible set and ${\rm L}_C(\emptyset)=\{0\}$. On the other hand, we have that $X$ is a $C$-admissible set if and only if $X\setminus \{0\}$ is a $C$-admissible set, and that ${\rm L}_C(X) = {\rm L}_C(X\setminus \{0\})$. All these considerations allow us to focus on the computation of ${\rm L}_C(X)$ when $X$ is a non-empty finite set of positive integers contained in $\{\theta(C),\to\}$.

\begin{proposition}\label{prop17}
Let $X=\{x_1,\ldots,x_t\}$ be a set of positive integers contained in $\{\theta(C),\to\}$ and let us suppose that $C=\{c_1,\ldots,c_q\}$. Then 
	$${\rm L}_C(X)= \left\{ a_1x_1+\cdots+a_tx_t+b_1c_1+\cdots+b_qc_q \mid a_1,\ldots,a_t,b_1,\ldots,b_q \in {\mathbb N} \right.$$
	$$\left. \mbox{ and } a_1+\cdots+a_t > b_1+\cdots+b_q \right\} \cup \{0\}.$$
\end{proposition}

\begin{proof}
Let $A= \{ a_1x_1+\cdots+a_tx_t+b_1c_1+\cdots+b_qc_q \mid a_1,\ldots,a_t,b_1,\ldots,b_q \in {\mathbb N} \mbox{ and } a_1+\cdots+a_t > b_1+\cdots+b_q \} \cup \{0\}$. Since $a_1x_1+\cdots+a_tx_t+b_1c_1+\cdots+b_qc_q \geq (a_1+\cdots+a_t)\theta(C)-(b_1+\cdots+b_q)\theta(C) \geq \theta(C)$, then we have that $A \subseteq {\mathbb N}$.

It easy to show that $A$ is closed for the addition, $0 \in A$, and that, if $x,y \in A \setminus \{0\}$, then $\{x+y\} + C \subseteq A$. Therefore, $A$ is a $C$-incentive.

Now, it is obvious that $X \subseteq A$ and, consequently, ${\rm L}_C(X) \subseteq A$. Thus, in order to finish the proof, it is enough to show that $A \subseteq {\rm L}_C(X)$. Thereby, let $x=a_1x_1+\cdots+a_tx_t+b_1c_1+\cdots+b_qc_q \in A$ and let us apply induction over $b_1+\cdots+b_q$ to prove that $x \in {\rm L}_C(X)$.

If $b_1+\cdots+b_q=0$, then $x=a_1x_1+\cdots+a_tx_t \in {\rm L}_C(X)$. Thus, we can suppose that $b_1+\cdots+b_q\geq 1$ and, consequently, $a_1+\cdots+a_t\geq2$. Thereby there exist $j \in \{1,\ldots,q\}$ and $i \in \{1,\ldots,t\}$ such that $b_j\not=0$ and $a_i\not=0$. By hypothesis of induction, we have that $x-x_i-c_j \in {\rm L}_C(X)$. Moreover, since $a_1+\cdots+a_t\geq2$, we deduce that $x-x_i-c_j \not= 0$. Now, applying that ${\rm L}_C(X)$ is a $C$-incentive, we have that $\{x-x_i-c_j\}+\{x_i\}+C \subseteq {\rm L}_C(X)$. Therefore, $x\in {\rm L}_C(X)$.
\end{proof}

Let us illustrate the above results with several examples.

\begin{example}\label{exmp18a}
Let us compute ${\rm L}_{\{-3,2\}}(\{5,7,9,11\})$. By Proposition~\ref{prop17}, since $\theta(\{-3,2\})=3$ and $\{5,7,9,11\} \subseteq \{3,\to\}$, we have that ${\rm L}_{\{-3,2\}}(\{5,7,9,11\})=\{a_1 5+a_2 7+a_3 9+a_4 11+b_1 (-3)+b_2 2 \mid a_1,a_2,a_3,a_4,b_1,b_2$$\in {\mathbb N} \mbox{ and } a_1+a_2+a_3+a_4>b_1+b_2\} \cup \{0\} = \{0,5,7,9,10,11,12,14,\to\}= \langle 5,7,9,11,13 \rangle$.
\end{example}

\begin{example}\label{exmp18b}
In order to compute ${\rm L}_{\{-4,6\}}(\{2,8\})$, observe that $\theta(\{-4,6\})=4$, $\{-4,6\} \subseteq \{k\,\frac{4}{2} \mid k\in\{-2,-1\}\cup {\mathbb N}\}$, and $\{2,8\} \subseteq \langle \frac{4}{2} \rangle = \langle 2 \rangle$. Now, by Proposition~\ref{prop14}, we have that $\{2,8\}$ is a $\{-4,6\}$-admissible set and, therefore, there exists ${\rm L}_{\{-4,6\}}(\{2,8\})$. From Theorem~\ref{thm16}, we conclude that ${\rm L}_{\{-4,6\}}(\{2,8\}) = \langle 2 \rangle$.
\end{example}

\begin{example}\label{exmp18c}
From Proposition~\ref{prop14}, we have that $\{3\}$ is not a $\{-4,6\}$-admissible set. Therefore, there does not exist ${\rm L}_{\{-4,6\}}(\{3\})$.
\end{example}

\begin{example}\label{exmp18d}
From Proposition~\ref{prop14}, we have that $\{2\}$ is not a $\{-4,7\}$-admissible set. Therefore, there does not exist ${\rm L}_{\{-4,7\}}(\{2\})$.
\end{example}

Now we are ready to show the algorithm that allows us to compute ${\rm L}_C(X)$ if $X$ is a non-empty finite set of positive integers contained in $\{\theta(C), \to \}$ (such as in Example~\ref{exmp18a}). This algorithm provides us an alternative method to the one given in Proposition~\ref{prop17}. Moreover, its validity and correctness is justified by Proposition~\ref{prop8}.

\begin{algorithm}\label{alg19}
INPUT: A non-empty finite set $X \subseteq \{\theta(C),\to\}$. \par
OUTPUT: The minimal system of generators of ${\rm L}_C(X)$. \par
(1) $D=\emptyset$. \par
(2) $Y={\rm msg}(\langle X \rangle)$. \par
(3) $E=\{s+t \mid s,t \in Y\} \setminus D$. \par
(4) $Z=Y \cup \left( \bigcup_{e\in E} \{e\}+C\right)$.\par
(5) If ${\rm msg}(\langle Z \rangle)=Y$, then return $Y$. \par
(6) Set $Y={\rm msg}(\langle Z \rangle)$, $D=D\cup E$, and go to (3).\par
\end{algorithm}

Let us illustrate the performance of this algorithm with an example.

\begin{example}\label{exmp20}
Let us compute ${\rm L}_{\{-3,2\}}(\{5,7,9,11\})$.
\begin{itemize}
\item $D=\emptyset$.
\item $Y=\{5,7,9,11\}$.
\item $E=\{10,12,14,16,18,20,22\}$.
\item $Z=\{5,7,9,11,12,13,14,15,16,17,18,19,20,22,24\}$, \\ ${\rm msg}(\langle Z \rangle)=\{5,7,9,11,13\}$.
\item $Y=\{5,7,9,11,13\}$, \\ $D=\{10,12,14,16,18,20,22\}$.
\item $E=\{24,26\}$.
\item $Z=\{5,7,9,11,13,21,23,26,28\}$, \\ ${\rm msg}(\langle Z \rangle)=\{5,7,9,11,13\}$.
\item $Y=\{5,7,9,11,13\}$.
\end{itemize}
Therefore, ${\rm L}_{\{-3,2\}}(\{5,7,9,11\})=\langle 5,7,9,11,13 \rangle$.
\end{example}

Observe that the most complex process in Algorithm~\ref{alg19} is the computation of ${\rm msg}(\langle Z \rangle)$, that is, compute the minimal system of generators of a monoid $M$ starting from any system of generators of it. For this purpose, we can use the \texttt{GAP} package \texttt{numericalsgps} (see \cite{numericalsgps}).

\section{Numerical $C$-incentives}\label{numerical-incentives}

Let $M$ be a $C$-incentive. We will say that $M$ is \emph{numerical} (that is, $M$ is a \emph{numerical $\,C$-incentive}) if $\gcd(M)=1$ (or, equivalently, if ${\mathbb N} \setminus M$ is a finite set). The purpose of this section is to show that, for the study of $C$-incentives, we can focus on numerical $C$-incentives.

Along this section, we will suppose that $C=\{c_1,\ldots,c_q\}$ is a non-empty subset of ${\mathbb Z}$ and denote by ${\rm NI}(C)=\{M \in {\rm I}(C) \mid M \mbox{ is numerical}\;\!\}$.

\begin{proposition}\label{prop21}
Let $X$ be a non-empty subset of $\{\theta(C),\to\}\setminus \{0\}$. Then ${\rm L}_C(X)$ is a numerical semigroup if and only if $\gcd(X\cup C)=1$.
\end{proposition}

\begin{proof}
(Necessity.) Let us suppose that $\gcd(X\cup C)=d\not=1$. Then it is clear that $M=\{kd \mid kd \geq \theta(C) \} \cup \{0\}$ is a $C$-incentive containing $X$ and, therefore, ${\rm L}_C(X) \subseteq M$. Since ${\mathbb N} \setminus M$ is not a finite set, then ${\mathbb N} \setminus {\rm L}_C(X)$ is not finite and, consequently, ${\rm L}_C(X)$ is not a numerical semigroup.

(Sufficiency.) Let $A=X \cup (2X+C)$. It is clear that $A \subseteq {\rm L}_C(X)$. On the other hand, if $x\in X$, then $\gcd\{x,2x+c_1,\ldots,2x+c_q\}=\gcd\{x,c_1,\ldots,c_q\}$. Thereby, $\gcd(A)=1$. Consequently, $\gcd({\rm L}_C(X))=1$, that is, ${\rm L}_C(X)$ is a numerical semigroup.
\end{proof}

\begin{corollary}\label{cor22}
If $\gcd(C)=1$, then ${\rm I}(C)={\rm NI}(C) \cup \{\{0\}\}$.
\end{corollary}

\begin{proof}
First of all, let us observe that, if $\gcd(C)=1$ and we are in the case 2 of Theorem~\ref{thm16}, then $\frac{\theta(C)}{2}=1$. Therefore, $\langle \frac{\theta(C)}{2} \rangle = {\mathbb N}$ that is a numerical semigroup.

In any other case, the conclusion follows easily from Proposition~\ref{prop21} and Theorem~\ref{thm16}.
\end{proof}

Now we want to study the case $\gcd(C)\not=1$. Firstly we need two lemmas.

\begin{lemma}\label{lem23}
Let $M$ be a $C$-incentive such that $M\not= \{0\}$. Then $\gcd(M)$ divides $\gcd(C)$.
\end{lemma}

\begin{proof}
Let $x \in M\setminus \{0\}$. Then $\{x,2x+c_1,\ldots,2x+c_q\} \subseteq M$ and, therefore, $\gcd(M) | \gcd\{x,2x+c_1,\ldots,2x+c_q\}$. Now, being that $\gcd\{x,2x+c_1,\ldots,2x+c_q\} = \gcd\{x,c_1,\ldots,c_q\}$ and $\gcd\{x,c_1,\ldots,c_q\} | \gcd\{c_1,\ldots,c_q\}$, we conclude that $\gcd(M) | \gcd(C)$.
\end{proof}

\begin{lemma}\label{lem24}
Let $M$ be a submonoid of $({\mathbb N},+)$ such that $M\not= \{0\}$ and let $d=\gcd(M)$. Then $M$ is a $C$-incentive is and only if $\frac{M}{d}$ is a $\frac{C}{d}$-incentive.
\end{lemma}

\begin{proof}
(Necessity.) If $x,y \in \frac{M}{d} \setminus \{0\}$, then $dx,dy \in M \setminus \{0\}$. Since $M$ is a $C$-incentive, then $\{dx+dy\}+C \subseteq M$. From Lemma~\ref{lem23}, we know that $d | \gcd(C)$ and, consequently, $\{x+y\} + \frac{C}{d} \subseteq \frac{M}{d}$. Therefore, $\frac{M}{d}$ is a $\frac{C}{d}$-incentive.

(Sufficiency.) If $a,b \in M \setminus \{0\}$, then $\frac{a}{d}, \frac{b}{d} \in \frac{M}{d} \setminus \{0\}$. Since $\frac{M}{d}$ is a $C$-incentive, then $\{\frac{a}{d} + \frac{b}{d}\} + \frac{C}{d} \subseteq \frac{M}{d}$ and, therefore, $\{a+b\}+C \subseteq M$. In this way, $M$ is a $C$-incentive.
\end{proof}

\begin{theorem}\label{thm25}
Let $D$ be the set of all positive divisors of $\gcd(C)$. Then we have that ${\rm I}(C) \setminus \{0\} = \bigcup_{d \in D} \left\{dS \mid S \in {\rm NI}\left(\frac{C}{d}\right)\right\}$.
\end{theorem}

\begin{proof}
Let $M \in {\rm I}(C)$ such that $M \not= \{0\}$ and $\gcd(M)=d$. Then, by applying Lemmas~\ref{lem23} and \ref{lem24}, it is clear that $d \in D$ and $\frac{M}{d} \in {\rm NI}\left(\frac{C}{d}\right)$. For the other inclusion, by Lemma~\ref{lem24}, if $d \in D$ and $S \in {\rm NI}\left(\frac{C}{d}\right)$, then $dS \in {\rm I}(C)$.
\end{proof}

Let us illustrate the content of the previous theorem with an example.

\begin{example}\label{exmp26}
By applying Theorem~\ref{thm25}, we have that
	$${\rm I}(\{-4,6\}) = \{S \mid S \in {\rm NI}(\{-4,6\})\} \cup \{2S \mid S \in {\rm NI}(\{-2,3\})\} \cup \{\{0\}\}.$$
Thus, in order to compute ${\rm I}(\{-4,6\})$, it is enough to calculate ${\rm NI}(\{-4,6\})$ and ${\rm NI}(\{-2,3\})$.
\end{example}

We finish this section showing that, if we want to compute ${\rm L}_C(X)$, then we can focus on the case in which $\gcd(X \cup C)=1$.

\begin{lemma}\label{lem27}
Let $X$ be a set of positive integers such that $\gcd(X \cup C)=d$. Then $X$ is $C$-admissible if and only if $\frac{X}{d}$ is $\frac{C}{d}$-admissible.
\end{lemma}

\begin{proof}
It is a consequence of Proposition~\ref{prop14}, having in mind both of the following facts.
	\begin{enumerate}
	\item $X \subseteq \{\theta(C), \to\}$ if and only if $\frac{X}{d} \subseteq \left\{ \theta\left( \frac{C}{d} \right),\to \right\}$.
	\item $X \subseteq \left\langle \frac{\theta(C)}{2}\right\rangle$ and $C \subseteq  \left\{k\frac{\theta(C)}{2} \mid  k \in \{-2,-1\} \cup {\mathbb N} \right\}$ if and only if $\frac{X}{d} \subseteq \left\langle \frac{\theta\left( \frac{C}{d} \right)}{2}\right\rangle$ and $\frac{C}{d} \subseteq  \left\{k\frac{\theta\left( \frac{C}{d} \right)}{2} \mid  k \in \{-2,-1\} \cup {\mathbb N} \right\}$. \qedhere
	\end{enumerate}
\end{proof}

\begin{proposition}\label{prop28}
Let $X$ be a $C$-admissible set such that $\gcd(X \cup C)=d$. Then $\frac{X}{d}$ is $\frac{C}{d}$-admissible and ${\rm L}_C(X)=d\cdot {\rm L}_{\frac{C}{d}} \left(\frac{X}{d}\right)$.
\end{proposition}

\begin{proof}
By Lemma~\ref{lem27} and Proposition~\ref{prop17}, we have that, if $X \subseteq \{0,\theta(C), \to\}$, then ${\rm L}_C(X)=d\cdot {\rm L}_{\frac{C}{d}} \left(\frac{X}{d}\right)$.

On the other hand, by Proposition~\ref{prop14}, if $X \not\subseteq \{0,\theta(C), \to\}$, then $X \subseteq \left\langle \frac{\theta(C)}{2}\right\rangle$ and $C \subseteq  \left\{k\frac{\theta(C)}{2} \mid  k \in \{-2,-1\} \cup {\mathbb N} \right\}$. Thereby, by applying Theorem~\ref{thm16}, we have that ${\rm L}_C(X)=\left\langle \frac{\theta(C)}{2}\right\rangle$. Moreover, by Lemma~\ref{lem27} and Theorem~\ref{thm16}, we get that ${\rm L}_{\frac{C}{d}} \left(\frac{X}{d}\right) = \left\langle \frac{\theta\left( \frac{C}{d} \right)}{2}\right\rangle$. Therefore, ${\rm L}_C(X)=d\cdot {\rm L}_{\frac{C}{d}} \left(\frac{X}{d}\right)$.
\end{proof}

Let us illustrate the content of the above proposition with an example.

\begin{example}\label{exmp29}
Let us take the sets $C=\{-2,2\}$ and $X=\{4,6\}$. Then $\theta(C)=2$ and $X\subseteq \{0,\theta(C),\to\}$. By applying Proposition~\ref{prop14}, we have that $X$ is $C$-admissible. Since $\gcd(X \cup C)=2$, by Proposition~\ref{prop28}, then we have that $\{2,3\}$ is $\{-1,1\}$-admissible and that ${\rm L}_C(X) = 2\cdot {\rm L}_{\{-1,1\}}(\{2,3\})$. Now, from Proposition~\ref{prop8}, we easily deduce that $\langle 2,3 \rangle$ is a $\{-1,1\}$-incentive and, therefore, that ${\rm L}_{\{-1,1\}}(\{2,3\})=\langle 2,3 \rangle$. Consequently, ${\rm L}_C(X)=2\cdot \langle 2,3 \rangle = \langle 4,6 \rangle$.
\end{example}

\section{The Frobenius pseudo-variety of the numerical $C$-incentives}\label{pseudo-variety}

Let $S$ be a numerical semigroup. The \emph{Frobenius number} of $S$, denoted by ${\rm F}(S)$, is the greatest integer that does not belong to $S$ (see \cite{alfonsin}).

A \emph{Frobenius pseudo-variety} is a non-empty family  ${\mathcal P}$ of numerical semigroups that fulfills the following conditions.
\begin{enumerate}
	\item ${\mathcal P}$ has a maximum element $\max({\mathcal P})$ (with respect to the inclusion order).
	\item if $S,T \in {\mathcal P}$, then $S\cap T\in {\mathcal P}$.
	\item if $ S\in {\mathcal P}$ and $S\neq \max({\mathcal P})$, then $S\cup \{{\rm F}(S)\} \in {\mathcal P}$.
\end{enumerate}

Let us observe that a Frobenius pseudo-variety ${\mathcal P}$ is a Frobenius variety if and only if ${\mathbb N} \in {\mathcal P}$ (see \cite[Proposition 1]{pseudovar}).

Along this section, $C$ denotes a non-empty finite subset of ${\mathbb Z}$. Our purpose will be to show that ${\rm NI}(C)$ is a Frobenius pseudo-variety.

\begin{lemma}\label{lem30}
If $r \in {\mathbb N}$, then ${\rm NI}(\{r\})$ is a Frobenius pseudo-variety and, moreover, $\max({\rm NI}(\{r\}))={\mathbb N}$.
\end{lemma}

\begin{proof}
It is clear that ${\mathbb N} \in {\rm NI}(\{r\})$ and, therefore, that $\max({\rm NI}(\{r\}))={\mathbb N}$. Also, it is easy to see that, if $S,T \in {\rm NI}(\{r\})$, then $S\cap T\in {\rm NI}(\{r\})$. Finally, let us take $S \in {\rm NI}(\{r\}) \setminus {\mathbb N}$ and see that $S \cup \{{\rm F}(S)\} \in {\rm NI}(\{r\})$. Indeed, let $x,y \in (S \cup \{{\rm F}(S)\}) \setminus \{0\}$. If $x,y \in S$, then $x+y+r \in S \subseteq S \cup \{{\rm F}(S)\}$. If ${\rm F}(S) \in \{x,y\}$, then $x+y+r \geq {\rm F}(S)$ and, consequently, $x+y+r \in S \cup \{{\rm F}(S)\}$.
\end{proof}

As a consequence of the previous lemma, we can observe that, if $r \in {\mathbb N}$, then ${\rm NI}(\{r\})$ is a Frobenius variety, since ${\mathbb N} \in {\rm NI}(\{r\})$.

\begin{lemma}\label{lem31}
If $r$ is a positive integer, then ${\rm NI}(\{-r\})$ is a Frobenius pseudo-variety. Moreover,
	$$\max({\rm NI}(\{-r\}))= \left\{ \begin{array}{l} {\mathbb N}, \; \mbox{ if } r \in \{1,2\}, \\[1mm] \{0,r,\to\}, \; \mbox{ if } r\geq 3. \end{array} \right.$$
\end{lemma}

\begin{proof}
It is clear that, if $r \in \{1,2\}$, then ${\mathbb N} \in {\rm NI}(\{-r\})$ and, therefore, that $\max({\rm NI}(\{-r\}))={\mathbb N}$. From Lemmas~\ref{lem11} and \ref{lem12}, we easily deduce that, if $r \geq 3$, then $\max({\rm NI}(\{-r\}))=\{0,r,\to\}$ (observe that, if $r \geq 3$, then $\langle \frac{r}{2} \rangle$ is not a numerical semigroup). Also, it is not difficult to check that, if $S,T \in {\rm NI}(\{-r\})$, then $S\cap T\in {\rm NI}(\{-r\})$.

Now, let us see that, if $S \in {\rm NI}(\{-r\})$ and $S \not=\max({\rm NI}(\{-r\}))$, then $S \cup \{{\rm F}(S)\} \in {\rm NI}(\{-r\})$. In order to do this, let us take $x,y \in (S \cup \{{\rm F}(S)\}) \setminus \{0\}$. Firstly, if $x,y \in S$, then $x+y-r \in S \subseteq S \cup \{{\rm F}(S)\}$. Thus, we can suppose that ${\rm F}(S) \in \{x,y\}$. We distinguish two cases.
\begin{enumerate}
\item Let us suppose that $x={\rm F}(S)$ and $y\not={\rm F}(S)$. Then $y \in S \setminus \{0\}$ and, since $S \subsetneqq \max({\rm NI}(\{-r\}))$, we can deduce that $y \geq r$. Therefore, $x+y-r \geq {\rm F}(S)$ and, consequently, $x+y-r \in S \cup \{{\rm F}(S)\}$.

\item Let us suppose that $x=y={\rm F}(S)$. Then $x+y-r=2{\rm F}(S)-r$. We have two possibilities.
	\begin{enumerate}
	\item If ${\rm F}(S)\geq r$, then $2{\rm F}(S)-r\geq {\rm F}(S)$ and, therefore, $x+y-r \in S \cup \{{\rm F}(S)\}$.
	\item If ${\rm F}(S)<r$, since $S \subsetneqq \max({\rm NI}(\{-r\}))$, we have that $r=2$ and $S=\{0,2,\to\}$. Therefore, $S \cup \{{\rm F}(S)\} = {\mathbb N} \in {\rm NI}(\{-2\})$. \qedhere
	\end{enumerate}
\end{enumerate}
\end{proof}

Let us observe that, as a consequence of the previous lemma, we have that ${\rm NI}(\{-2\})$ and ${\rm NI}(\{-1\})$ are Frobenius varieties, since they contain ${\mathbb N}$. On the other hand, if $r\geq 3$, then ${\rm NI}(\{-r\})$ is a Frobenius pseudo-variety but not a Frobenius variety.

\begin{lemma}\label{lem32}
Let $\{{\mathcal P}_i\}_{i\in I}$ be a family of Frobenius pseudo-varieties. If there exists $j \in I$ such that $\max({\mathcal P}_j) \in {\mathcal P}_i$ for all $i \in I$, then $\bigcap_{i \in I} {\mathcal P}_i$ is a Frobenius pseudo-variety and $\max(\bigcap_{i \in I} {\mathcal P}_i)=\max({\mathcal P}_j)$.
\end{lemma}

\begin{proof}
It is clear that $\max(\bigcap_{i \in I} {\mathcal P}_i)=\max({\mathcal P}_j)$. Now, if $S,T \in \bigcap_{i \in I} {\mathcal P}_i$, then $S,T \in  {\mathcal P}_i$ for all $i \in I$ and, therefore, $S \cap T \in \bigcap_{i \in I} {\mathcal P}_i$. Finally, if $S \in \bigcap_{i \in I} {\mathcal P}_i$ and $S\not=\max(\bigcap_{i \in I} {\mathcal P}_i)$, then $S\in {\mathcal P}_i$ and $S\not=\max({\mathcal P}_i)$ for all $i\in I$. Therefore, $S \cup {\rm F}(S) \in {\mathcal P}_i$ for all $i \in I$. Consequently, $S \cup {\rm F}(S) \in \bigcap_{i \in I} {\mathcal P}_i$.
\end{proof}

An immediate consequence of Lemma~\ref{lem11} is the next one.

\begin{lemma}\label{lem33}
$\{0,\theta(C),\to\} \in {\rm NI}(\{c\})$ for all $c \in C$.
\end{lemma}

We are ready to show the main result of this section.

\begin{theorem}\label{thm34}
${\rm NI}(C)$ is a Frobenius pseudo-variety. Moreover,
	$$\max({\rm NI}(C))= \left\{ \begin{array}{l} {\mathbb N}, \; \mbox{ if } C \subseteq \{-2,-1\} \cup {\mathbb N}, \\[1mm] \{0,\theta(C),\to\}, \; \mbox{ in other case. } \end{array} \right.$$
\end{theorem}

\begin{proof}
It is clear that ${\rm NI}(C) = \bigcap_{c\in C} {\rm NI}(\{c\})$. From Lemmas~\ref{lem30} and \ref{lem31}, we know that ${\rm NI}(\{c\})$ is a Frobenius pseudo-variety for all $c \in C$, and that, if $C \subseteq \{-2,-1\} \cup {\mathbb N}$, then $\max({\rm NI}(\{c\})) = {\mathbb N}$ for all $c \in C$. Thus, from Lemma~\ref{lem32}, we have that ${\rm NI}(C)$ is a Frobenius pseudo-variety with $\max({\rm NI}(C))={\mathbb N}$. Now, let us suppose that $\theta(C)\geq 3$ and let $c_0\in C$ such that $\theta(C)=-c_0$. From Lemma~\ref{lem31}, we have that $\max({\rm NI}(\{c_0\}))=\{0,\theta(C),\to\}$ and, from Lemma~\ref{lem33}, that $\max({\rm NI}(\{c_0\})) \in {\rm NI}(\{c\})$ for all $c \in C$. Therefore, by applying Lemma~\ref{lem32}, ${\rm NI}(C)$ is a Frobenius pseudo-variety with $\max({\rm NI}(C))=\{0,\theta(C),\to\}$.
\end{proof}

\begin{remark}
If $r$ is a positive integer different from two, then $\max({\rm NI}(\{-r\}))=\{0,r,\to\}$. Moreover, as a consequence of Theorem~\ref{thm34}, $\max({\rm NI}(\{-r\}))\not=\{0,2,\to\}$ for all $C \subseteq {\mathbb Z}$. Consequently, we conclude that, if $\{0,2,\to\} \in {\rm NI}(C)$, then ${\mathbb N} \in {\rm NI}(C)$.
\end{remark}

\begin{remark}
From Theorem~\ref{thm34}, ${\rm NI}(C)$ is a Frobenius variety if and only if $C \subseteq \{-2,-1\} \cup {\mathbb N}$. Several of these families have been studied in some previous works. For instance, ${\rm NI}(\{1\})$, ${\rm NI}(\{-1\})$, ${\rm NI}(\{-1,1\})$, and ${\rm NI}(C)$ (for $C \subseteq {\mathbb N}$) are analysed in \cite{frases}, \cite{digitales}, \cite{acotados,benefits}, and \cite{brazaletes}, respectively.
\end{remark}

\section{The tree of the numerical $C$-incentives}\label{tree}

Our purpose in this section will be to arrange the elements of ${\rm NI}(C)$ in a tree with root and to characterize the children in such a tree. Thus, as main result of this paper, we will obtain an algorithmic process that will allow us to recursively build the elements of ${\rm NI}(C)$.

Recall that a \emph{graph} $G$ is a pair $(V,E)$, where
\begin{itemize}
\item $V$ is a non-empty set, which elements are called \emph{vertices} of $G$,
\item $E$ is a subset of $\{(v,w) \in V \times V \mid v \neq w\}$, which elements are called \emph{edges} of $G$.
\end{itemize}
A \emph{path (of length $n$)} connecting the vertices $x$ and $y$ of $G$ is a sequence of different edges of the form $(v_0,v_1),(v_1,v_2),\ldots,(v_{n-1},v_n)$ such that $v_0=x$ and $v_n=y$. Moreover, we say that a graph $G$ is a \emph{tree} if there exists a vertex $v^*$ (known as the \emph{root} of $G$) such that, for every other vertex $x$ of $G$, there exists a unique path connecting $x$ and $v^*$. If $(x,y)$ is an edge of the tree, then we say that $x$ is a \emph{child} of $y$.

In this section, we will suppose that $C$ is a non-empty finite set of ${\mathbb Z}$. We define the graph ${\rm G}(C)$ in the following way.
\begin{itemize}
	\item ${\rm NI}(C)$ is the set of vertices of ${\rm G}(C)$;
	\item $(S,S')\in {\rm NI}(C)\times {\rm NI}(C)$ is an edge of ${\rm G}(C)$ if $S'=S\cup\{{\rm F}(S)\}$.
\end{itemize}
It is well known (see \cite{springer}) that, if $M$ is a submonoid of $({\mathbb N},+)$ and $x\in M$, then $M\setminus \{x\}$ is a monoid if and only if $x\in {\rm msg}(M)$. As a consequence of \cite[Lemma 12, Theorem 3]{pseudovar}, we have the next result.
 
\begin{theorem}\label{thm35}
The graph ${\rm G}(C)$ is a tree with root equal to $\max({\rm NI}(C))$. Moreover, the children of a vertex $S \in {\rm NI}(C)$ are the elements of the set
$$\left\{ S \setminus \{x\} \mid x \in {\rm msg}(S), \; x > {\rm F}(S), \mbox{ and } S \setminus \{x\} \in {\rm NI}(C) \right\}.$$
\end{theorem} 
 
In the following proposition we will characterize the minimal generators $x$ of a $C$-incentive $M$ such that $M\setminus \{x\}$ is also a $C$-incentive.

\begin{proposition}\label{prop36}
Let $M$ be a $C$-incentive and $x\in {\rm msg}(M)$. Then $M\setminus \{x\}$ is a $C$-incentive if and only if $\{x\}-C \subseteq ({\mathbb Z}\setminus M) \cup {\rm msg}(M\setminus \{x\}) \cup \{x,0\}$.
\end{proposition}
 
\begin{proof}
(Necessity.) If $x-c \notin ({\mathbb Z}\setminus M) \cup {\rm msg}(M\setminus \{x\}) \cup \{x,0\}$ for some $c\in C$, then we can assert that $x-c \in M\setminus \{x,0\}$ and $x-c \notin {\rm msg}(M\setminus \{x\})$. Therefore, $x-c=m+n$ for some $m,n \in M\setminus \{x,0\}$ and, consequently, $m+n+c=x \notin M \setminus \{x\}$. Thereby, $ M \setminus \{x\}$ is not a $C$-incentive.

(Sufficiency.) If we take $m,n \in M\setminus \{x,0\}$, then $\{m+n\} +C \subseteq M$. Let us suppose that $m+n+c=x$ for some $c\in C$. In such a case, $x-c \notin ({\mathbb Z}\setminus M) \cup {\rm msg}(M\setminus \{x\}) \cup \{x,0\}$ that is a contradiction. Thus, $m+n+c\not=x$ for all $c\in C$ and, consequently, $\{m+n\} +C \subseteq M\setminus \{x\}$.
\end{proof}

In order to facilitate the construction of the tree ${\rm G}(C)$, we will study the relation between the minimal generators of a numerical semigroup $S$ and the minimal generators of $S\setminus \{x\}$, where $x$ is a minimal generator of $S$ that is greater than ${\rm F}(S)$. First of all, let us observe that, if $S$ is minimally generated by $\{m,m+1,\ldots,2m-1\}$ (that is, $S=\{0,m,\to\}$), then $S\setminus \{m\}$ is minimally generated by $\{m+1,m+2,\ldots,2m+1\}$. In other case we will use the next result, which is a reformulation of \cite[Corollary 18]{frases}.

\begin{proposition}\label{prop37}
Let $S$ be a numerical semigroup with minimal system of generators $\left\{n_1<\ldots<n_p\right\}$. If $i \in \{2,\ldots,p\}$ and $n_i>{\rm F}(S)$, then
	$${\rm msg}(S \setminus \{n_i\})= \left\{ \begin{array}{l}
			\left\{n_1,\ldots,n_{p}\right\} \setminus \{n_i\}, \quad \mbox{if there exists } j \in \{2,\ldots,p-1\} \\ \mbox{ } \hspace{3.27cm} \mbox{such that } n_i+n_1-n_j \in S; \\[2mm]
			\left(\left\{n_1,\ldots,n_{p}\right\} \setminus \{n_i\}\right) \cup \left\{n_i+n_1\right\}, \quad \mbox{in other case.}
	\end{array} \right.$$
\end{proposition}

Let $S$ be a numerical $C$-incentive, $m=\min\left({\rm msg}(S)\right)$, and $x\in {\rm msg}(S)$. From Proposition~\ref{prop37}, if $(\{x\}-C) \cap \left( {\rm msg}(M\setminus \{x\}) \setminus {\rm msg}(M) \right) \not= \emptyset$, then $-m \in C$. On the other hand, $x \in \{x\}-C$ if and only if $0 \in C$. These two facts allow us to give the following improvement of Proposition~\ref{prop36} (see Remark~\ref{rem38}).

\begin{proposition}\label{prop36b}
Let $S$ be a numerical $C$-incentive, $m=\min\left({\rm msg}(S)\right)$, and $x\in {\rm msg}(S)$. Let us suppose that $-m \not\in C$. Then $S \setminus \{x\}$ is a numerical $C$-incentive if and only if $\{x\}-C \subseteq ({\mathbb Z}\setminus S) \cup {\rm msg}(S) \cup \{x\}$.
\end{proposition}

\begin{remark}\label{rem38}
Observe that, by applying Proposition~\ref{prop36}, we have to compute ${\rm msg}(S\setminus \{x\})$ in order to assert that $S\setminus \{x\}$ is a numerical $C$-incentive. That is, firstly we compute and secondly we assert. However, by Proposition~\ref{prop36b}, we have only to use ${\rm msg}(S)$. Of course, if we want to build the tree, we will have to compute ${\rm msg}(S\setminus \{x\})$. Now, firstly we assert and secondly we compute. 
\end{remark}

Let us see an example that illustrates the contents of this section.

\begin{example}\label{exmp38}
We are going to build the tree associated to the numerical $\{-3,2\}$-incentives.
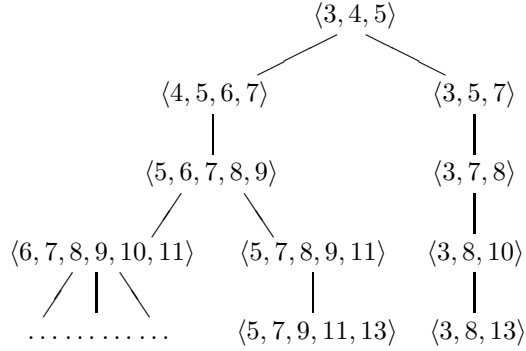
\begin{center}
\begin{picture}(205,130)
%\put(0,0){A} \put(0,130){B} \put(205,0){C} \put(205,130){D} \put(0,30){E} \put(205,60){F}
\put(123,120){$\langle 3,4,5 \rangle$}
\put(132,115){\line(-2,-1){30}} \put(143,115){\line(2,-1){30}}
\put(65,90){$\langle 4,5,6,7 \rangle$} \put(168,90){$\langle 3,5,7 \rangle$}
\put(85,85){\line(0,-1){15}} \put(184,85){\line(0,-1){15}}
\put(59,60){$\langle 5,6,7,8,9 \rangle$} \put(168,60){$\langle 3,7,8 \rangle$}
\put(73,55){\line(-2,-3){11}} \put(97,55){\line(2,-3){11}} \put(184,55){\line(0,-1){15}}
\put(8,30){$\langle 6,7,8,9,10,11 \rangle$} \put(95,30){$\langle 5,7,8,9,11 \rangle$} \put(166,30){$\langle 3,8,10 \rangle$}
\put(32,25){\line(-2,-3){11}} \put(41,25){\line(0,-1){15}} \put(50,25){\line(2,-3){11}} \put(123,25){\line(0,-1){15}} \put(184,25){\line(0,-1){15}}
\put(15,0){\ldots} \put(29,0){\ldots} \put(43,0){\ldots} \put(57,0){\ldots} \put(94,0){$\langle 5,7,9,11,13 \rangle$} \put(167,0){$\langle 3,8,13 \rangle$}
\end{picture}
\end{center}

\medskip

By Theorem~\ref{thm34}, we know that $\max({\rm NI}(\{-3,2\}))=\{0,3,\to\} = \langle 3,4,5 \rangle$.  By applying Theorem~\ref{thm35} and Propositions~\ref{prop36}, \ref{prop37}, and \ref{prop36b} (in fact, we will apply Proposition~\ref{prop36} only when $\min\left({\rm msg}(S)\right)=3$), we have that
\begin{itemize}
\item $\langle 4,5,6,7 \rangle = \langle 3,4,5 \rangle \setminus \{3\}$ and $\langle 3,5,7 \rangle = \langle 3,4,5 \rangle \setminus \{4\}$ are the two children of $\langle 3,4,5 \rangle$.
\item $\langle 5,6,7,8,9 \rangle = \langle 4,5,6,7 \rangle \setminus \{4\}$ is the unique child of $\langle 4,5,6,7 \rangle$.
\item $\langle 3,7,8 \rangle = \langle 3,5,7 \rangle \setminus \{5\}$ is the unique child of $\langle 3,5,7 \rangle$.
\item $\langle 6,7,8,9,10,11 \rangle = \langle 5,6,7,8,9 \rangle \setminus \{5\}$ and $\langle 5,7,8,9,11 \rangle = \langle 5,6,7,8,9 \rangle \setminus \{6\}$ are the two children of $\langle 5,6,7,8,9 \rangle$.
\item $\langle 3,8,10 \rangle = \langle 3,7,8 \rangle \setminus \{7\}$ is the unique child of $\langle 3,7,8 \rangle$.
\item $\langle 6,7,8,9,10,11 \rangle$ has three children.
\item $\langle 5,7,9,11,13 \rangle = \langle 5,7,8,9,11 \rangle \setminus \{8\}$ is the unique child of $\langle 5,7,8,9,11 \rangle$.
\item $\langle 3,8,13 \rangle = \langle 3,8,10 \rangle \setminus \{10\}$ is the unique child of $\langle 3,8,10 \rangle$.
\item $\langle 5,7,9,11,13 \rangle$ has not got any child.
\item $\langle 3,8,13 \rangle$ has not got any child.
\item And so on.
\end{itemize}
\end{example}

\begin{remark}
In Example~\ref{exmp38} we have an infinite tree, that is, a tree with infinitely many elements. For instance, the branch $\langle 3,4,5 \rangle$, $\langle 4,5,6,7 \rangle$, $\langle 5,6,7,8,9, \rangle$, $\langle 6,7,8,9,10,11 \rangle$, $\ldots$ has no end. However, if we only take into account numerical $C$-incentives with Frobenius number (or genus) less than or equal to a fixed number, then we are going to obtain trees with finitely many elements. (Remind that, if $S$ is a numerical $C$-incentive, then the genus of $S$ is equal to the cardinality of ${\mathbb N} \setminus S$.)
\end{remark}

\section{$C$-incentives containing a given $C$-admissible set}\label{suggestion}

Let $X$ be a subset of ${\mathbb N} \setminus \{0\}$ such that it is $C$-admissible. We denote by ${\rm I}(C,X) = \{ M \in {\rm I}(C) \mid X \subseteq M \}$ and by ${\rm NI}(C,X) = \{ S \in {\rm NI}(C) \mid X \subseteq S \}$. In this section, our main purpose will be to show an algorithmic process that will allow us to compute ${\rm I}(C,X)$. In order to do that, and first of all, we will see that we can focus on the computation of ${\rm NI}(C,X)$.

\begin{lemma}\label{lem71}
If $M \in {\rm I}(C,X)$, then $\gcd(M)$ divides $\gcd(C \cup X)$.
\end{lemma}

\begin{proof}
By Lemma~\ref{lem23}, we know that $\gcd(M)$ divides $\gcd(C)$. In addition, since $X \subseteq M$, we have that $\gcd(M)$ divides $\gcd(X)$. Thus, $\gcd(M)$ divides $\gcd(C \cup X)$.
\end{proof}

As a consequence of the previous lemma, we have the following one.

\begin{lemma}\label{lem72}
If $\gcd(C \cup X)=1$, then ${\rm I}(C,X) = {\rm NI}(C,X)$.
\end{lemma}

Let $M$ be a submonoid of $({\mathbb N},+)$ such that $M \not= \{0\}$ and let $d=\gcd(M)$. In such a situation, from Lemma~\ref{lem24}, we have that $M \in {\rm I}(C)$ if and only if $\frac{M}{d} \in {\rm I}\left(\frac{C}{d}\right)$. Moreover, it is clear that $X \subseteq M$ if and only if $\frac{X}{d} \subseteq \frac{M}{d}$. In this way, we can establish the next result.

\begin{lemma}\label{lem73}
Let $M$ be a submonoid of $({\mathbb N},+)$ such that $M \not= \{0\}$ and let $d=\gcd(M)$. Then $M \in {\rm I}(C,X)$ if and only if $\frac{M}{d} \in {\rm I}\left(\frac{C}{d}, \frac{X}{d}\right)$.
\end{lemma}
 
Now, from Lemmas~\ref{lem71}, \ref{lem72}, and \ref{lem73}, we have the following result.

\begin{theorem}\label{thm74}
Let $D$ the set formed by all positive divisors of $\gcd(C \cup X)$. Then $${\rm I}(C,X) = \bigcup_{d \in D} \left\{ dS \mid S \in {\rm NI}\left(\frac{C}{d}, \frac{X}{d} \right) \right\}.$$
\end{theorem}

On the other hand, from Theorem~\ref{thm34}, we easily conclude the next one.

\begin{theorem}\label{thm75}
${\rm NI}(C,X)$ is a Frobenius pseudo-variety. Moreover,
	$$\max({\rm NI}(C,X))= \left\{ \begin{array}{l} {\mathbb N}, \; \mbox{ if } C \subseteq \{-2,-1\} \cup {\mathbb N}, \\[1mm] \{0,\theta(C),\to\}, \; \mbox{ in other case. } \end{array} \right.$$
\end{theorem}

Now, as is expected, we define the graph $G(C,X)$ in the following way.
\begin{itemize}
	\item ${\rm NI}(C,X)$ is the set of vertices of ${\rm G}(C,X)$;
	\item $(S,S')\in {\rm NI}(C,X)\times {\rm NI}(C,X)$ is an edge of ${\rm G}(C,X)$ if $S'=S\cup\{{\rm F}(S)\}$.
\end{itemize}

The next result is a direct consequence of Theorem~\ref{thm35}.

\begin{theorem}\label{thm76}
The graph ${\rm G}(C,X)$ is a tree with root equal to $\max({\rm NI}(C,X))$. Moreover, the children of a vertex $S \in {\rm NI}(C,X)$ are the elements of the set
$$\left\{ S \setminus \{a\} \mid a \in {\rm msg}(S), \; a > {\rm F}(S), S \setminus \{a\} \in {\rm NI}(C), \mbox{ and } a\not\in X \right\}.$$
\end{theorem}

By combining Theorems~\ref{thm75}, \ref{thm76}, and Propositions~\ref{prop36}, \ref{prop37}, and \ref{prop36b}, we can recursively build the tree $G(C,X)$ as shown in the next example.

\begin{example}\label{exmp77}
The tree associated to the numerical $\{-3,2\}$-incentives containing the set $\{5\}$ is the following one.
\begin{center}
\begin{picture}(155,130)
%\put(0,0){A} \put(0,130){B} \put(155,0){C} \put(155,130){D} \put(0,30){E} \put(155,90){F}
\put(73,120){$\langle 3,4,5 \rangle$}
\put(82,115){\line(-2,-1){30}} \put(93,115){\line(2,-1){30}}
\put(16,90){$\langle 4,5,6,7 \rangle$} \put(118,90){$\langle 3,5,7 \rangle$}
\put(36,85){\line(0,-1){15}}
\put(12,60){$\langle 5,6,7,8,9 \rangle$}
\put(36,55){\line(0,-1){15}}
\put(10,30){$\langle 5,7,8,9,11 \rangle$}
\put(36,25){\line(0,-1){15}}
\put(8,0){$\langle 5,7,9,11,13 \rangle$}
\end{picture}
\end{center}
In order to justify it, we can review the computations of Example~\ref{exmp38}.
\end{example}

In the previous example, we have obtained a finite tree. Indeed, this fact can be characterized in terms of $C$ and $X$.

\begin{theorem}\label{thm78}
${\rm NI}(C,X)$ is finite if and only if $\gcd(C \cup X)=1$.
\end{theorem}

\begin{proof}
(Necessity.) By applying Proposition~\ref{prop28}, if $\gcd(C \cup X)=d\not=1$, then we have that $\gcd({\rm L}_C(X))\not=1$. On the other hand, let us denote by $M_k={\rm L}_C(X) \cup \{k,\to \}$, for all $k \in {\mathbb N}$ such that $k \geq \theta(C)$. Since $M_k \in {\rm NI}(C,X)$, for all $k \in {\mathbb N}$ such that $k \geq \theta(C)$, then we conclude that ${\rm NI}(C,X)$ is infinite.

(Sufficiency.) From Lemma~\ref{lem72}, we know that, if $\gcd(C \cup X)=1$, then ${\rm L}_C(X) \in {\rm NI}(C,X)$. Since ${\rm L}_C(X)$ is contained in all elements of ${\rm NI}(C,X)$ and ${\mathbb N} \setminus {\rm L}_C(X)$ is finite, we easily deduce that ${\rm NI}(C,X)$ is finite.
\end{proof}

%\section*{Acknowledgement}

%The authors would like to thank the referee from providing constructive comments and help in improving the contents of this paper. In particular, the suggestion of the problem that it is studied in Section~\ref{suggestion}.

%\bibliographystyle{amsplain}

\end{document}